\documentclass[12pt]{amsart}

\usepackage{parskip, amsthm,amsmath,amssymb,amscd,verbatim}
\usepackage[margin=0.7in]{geometry}
\usepackage{amsfonts,  mathrsfs, stmaryrd}
\usepackage[all]{xy}
\usepackage{relsize}
\usepackage{graphicx}
\usepackage{hyperref}

\usepackage{enumitem,amssymb}
\newlist{todolist}{itemize}{2}
\setlist[todolist]{label=$\square$}
\usepackage{pifont}

  \theoremstyle{plain}
  \newtheorem{theorem}{Theorem}[section]
  \newtheorem{lemma}[theorem]{Lemma}

  \newtheorem{cor}[theorem]{Corollary}

  \theoremstyle{remark}

  \theoremstyle{definition}

  \numberwithin{equation}{section}
  \numberwithin{theorem}{section}

\newtheorem{proposition}[theorem]{Proposition}

\newcommand{\Z}{\mathbb{Z}}

\DeclareMathOperator{\Aut}{Aut}

\input xy
\xyoption{all}

\begin{document}


\author{Jonathan Cohen and Kyle Rosengartner}
\title{Finite Groups With Two Irredundant Covers}

\begin{abstract}
An irredundant cover of a finite group $G$ is a collection of proper subgroups whose union is $G$ and which contains no smaller subcover. We classify finite groups which possess exactly two irredundant covers, thereby initiating an answer to a question of Brodie, who classified finite groups with one irredundant cover. 
\end{abstract}

\maketitle

\section{Introduction}

Let $G$ be a group. A collection $\{H_i\}_{i=1}^n$ of proper subgroups of $G$ is called a {\it cover} if $G = \bigcup\limits_{i=1}^n H_i$. The smallest possible $n$ is called the {\it covering number} of $G$, typically denoted $\sigma(G)$. There has been significant progress in the understanding of covering numbers and their structural significance. For a nice survey of many of these and related results, see \cite{Bh}. 

In this article we consider a variation on this theme, in which the quantity of interest is not the size of the cover but the number of distinct covers. To be more precise, we say that a cover $\{H_i\}_{i=1}^n$ of $G$ is {\it irredundant} if it contains no proper subcover. In other words, $\bigcup\limits_{i\neq j} H_i \neq G$ for all $1\leq j\leq n$. Let $\beta(G)$ denote the number of distinct irredundant covers of a finite group $G$ (we will only consider finite groups). 

In \cite{B}, a classification of finite groups with $\beta(G)=1$ is obtained. At the end of that paper, the question is posed of extending such a determination for larger values of $\beta(G)$. The main result of this paper is such a classification for $\beta(G)=2$. To state it, we first give some basic notation: the group $\Z_n$ is the cyclic group of order $n$, $Q_8$ is the quaternion group of order $8$, and $K\rtimes H$ is a semidirect product of the groups $K$ and $H$. 

\begin{theorem}
Suppose that $G$ is a finite group with exactly two irredundant covers. 

A) If $G$ is nilpotent then $G$ is one of the following groups: 

i) $(\Z_{p^2} \rtimes \Z_p) \times \Z_n$ where $(p,n)=1$ and $p>2$ is prime.

ii) $\Z_4\times \Z_2\times \Z_n$ where $n$ is odd.  

B) If $G$ is not nilpotent then $G$ is one of the following groups:

iii) $\Z_{p^t}\rtimes \Z_n$ where $p$ is prime, $(p,n)=1$, $\Z_n$ centralizes $\Z_{p^{t-1}}$, and $\Z_n$ acts on $\Z_{p^t}$ by an automorphism of order $q^2$ for a prime $q$. 
 
 iv) $ \Z_p^2\rtimes \Z_n $ where $p$ is prime, $(p,n)=1$, and $\Z_n$ acts irreducibly on $\Z_p^2$ by an automorphism of prime order. 
 
 v) $ Q_8\rtimes \Z_n$ where $n$ is odd and $\Z_n$ acts on $Q_8$ by an automorphism of order 3. 
 

\end{theorem}

We remark that case iii) can occur if and only if $p-1$ is not square-free, while case iv) 
can occur if and only if $p+1$ is not a power of 2. 
We now outline the contents of this article. In section 2 we prove some preliminary results about irredundant covers. In section 3 we begin the classification of groups $G$ with $\beta(G)=2$, showing that it naturally breaks into two cases. In section 4 we handle the first case, which includes all nilpotent $G$, and in section 5 we handle the second case, which ends up including all nonnilpotent $G$. In section 6 we consider some possibilities for future work.

\section{Preliminaries}

Let $G$ be a finite group and let $\beta(G)$ denote the number of irredundant covers of $G$. We write $\beta(G)=0$ if $G$ is cyclic. If $N\leq  G$ then we say $N$ is {\it cyclically embedded} in $G$ if $\langle N, x\rangle$ is cyclic for all $x\in G$. Clearly any such $N$ is contained in $Z(G)$. A cyclic subgroup $C\leq G$ is a {\it maximal cyclic subgroup} if $C$ is not contained in any larger cyclic subgroup, and is a {\it cyclic maximal subgroup} if $C$ is not contained in any larger proper subgroup. Note that maximal cyclic subgroups always exist but cyclic maximal ones may not.

\begin{lemma}\label{basiclemma} Let $G$ be a finite noncyclic group and $N$ a normal subgroup. 

a) A subgroup appears in some irredundant cover of $G$ if and only it contains a maximal cyclic subgroup of $G$. In particular, any cyclic maximal subgroup appears in all irredundant covers. 

b) We have $\beta(G)\geq \beta(G/N)$, with equality if and only if $N$ is cyclically embedded. 

c) If $(|G|, n)=1$ then $\beta(G\times \Z_n) = \beta(G)$. 
\end{lemma}

\begin{proof}
a) Let $x_1, \ldots, x_n$ be generators of the distinct maximal cyclic subgroups of $G$. Clearly $G = \bigcup\limits_{i=1}^n \langle x_i \rangle$, and $\{ \langle x_i\rangle  \}_{i=1}^n$ is an irredundant cover. If $K_1, \ldots, K_m$ is a cover of $G$ then $G$ is the union of those $K_j$ that contain some $x_i$. So if the cover is irredundant then all $K_j$ contain some $x_i$. Conversely, if $H\supset \langle x_i\rangle$ for some $i$ then $G = H \cup \bigcup\limits_{x_i\not\in H} \langle x_i \rangle$ and this gives an irredundant cover since $x_i\in H\setminus  \bigcup\limits_{i\neq j} \langle x_j\rangle $. 

b) If $\{  K_i/N \}_{i=1}^n$ is an arbitrary irredundant cover of $G/N$ then $\{K_i\}_{i=1}^n$ is an irredundant cover of $G$. This determines a function $f$ from the set of irredundant covers of $G/N$ to those of $G$. Since $f$ is clearly injective, we have $\beta(G/N)\leq \beta(G)$. The image of $f$ is those irredundant covers $\{ K_i \}_{i=1}^n$ of $G$ with $K_i\supset N$ for $1\leq i\leq n$. We claim $f$ is surjective if and only if $N$ is cyclically embedded. 

Suppose first that $N\leq G$ is cyclically embedded. If $\langle x \rangle$ is a maximal cyclic subgroup then $\langle x, N \rangle$ is a cyclic subgroup containing $\langle x \rangle$, so $N\subset \langle x \rangle$. Thus $N$ is contained in all maximal cyclic subgroups, hence in every subgroup that appears in an irredundant cover by a). Thus $f$ is surjective and $\beta(G)= \beta(G/N)$.

Suppose now that $\beta(G)=\beta(G/N)$. Therefore if $\{ K_i \}_{i=1}^n$ is an irredundant cover of $G$ then $K_i\supset N$ for $1\leq i\leq n$. In particular this holds if the $K_i$ are the maximal cyclic subgroups of $G$. Let $g\in G$, and choose $x\in G$ so that $\langle x\rangle $ is a maximal cyclic subgroup containing $g$. Since $N\subset \langle x \rangle$ we have $\langle g, N \rangle \subset \langle x, N \rangle = \langle x \rangle$, and thus $N$ is cyclically embedded. 

c) The factor $\Z_n$ is cyclically embedded in $G$ so this follows from b). 
\end{proof}

\section{Groups with Two Irredundant Covers}

 Suppose $\beta(G)=2$. If $\{\langle x_i \rangle \}_{i=1}^n$ are the distinct maximal cyclic subgroups of $G$, this gives one irredundant cover. As proved above, if some $\langle x_i\rangle$ is a maximal subgroup of $G$ then it must be included in every irredundant cover. Since $\beta(G)>1$, we may assume without loss of generality that $\langle x_1\rangle$ is not maximal. Let $H\leq G$ with $\langle x_1 \rangle \subsetneq H \subsetneq G$. Then 
$\{ H\} \cup\{ \langle x_i \rangle : x_i \not\in H  \}$ is an irredundant cover since $x_1\in H\setminus  \bigcup\limits_{i=2}^n \langle x_i\rangle $. This therefore is the second irredundant cover of $G$. Since $H$ is the unique noncyclic subgroup that appears in an irredundant cover of $G$, $H$ is a characteristic subgroup of $G$. In particular, both covers of $G$ are normal covers: all conjugates of the subgroups appearing are also present. Since $H$ is maximal, being the unique proper subgroup containing $\langle x_1\rangle$, the index $[G:H]$ is prime. 
By symmetry, $H$ contains all the nonmaximal $\langle x_i \rangle $, so $\langle x_i\rangle \not\subset H$ if and only if $\langle x_i\rangle$ is a cyclic maximal subgroup. Thus $G$ contains cyclic (hence abelian) maximal subgroups, and in particular, $G$ is solvable. 




Fix $\langle x_i\rangle \not\subset H$. 
Since $\langle x_i\rangle $ is a maximal subgroup of the solvable group $G$, the index $[G: \langle x_i\rangle]=p^m$ for some prime $p$. Thus $\langle x_i\rangle$ contains a conjugate of each Sylow subgroup except the $p$-Sylows. In particular, all Sylow subgroups of $G$, with the possible exception of the $p$-Sylows, are cyclic. Let $A\subset \langle x_i\rangle$ be the maximal Hall subgroup of $\langle x_i\rangle$ with index a power of $p$. Then $A$ is also a Hall subgroup of $G$. By the proof of  [\cite{PK}, Theorem 1], either $G = A\rtimes P$ or $G=P\rtimes A$, where $P$ is some $p$-Sylow subgroup of $G$. 

Before proceeding to these two cases, we prove a useful fact about $D_{2n}$, the dihedral group of order $2n$. 

\begin{lemma}\label{dihedral}
Let $n\geq 2$. 

a) If $n$ is prime then $\beta(D_{2n})=1$. 

b) If $n$ is not prime then $\beta(D_{2n})>2$. 
\end{lemma} 

\begin{proof}
a) This follows from the classification given in \cite{B}, or from the fact that, when $n$ is prime, every nontrivial proper subgroup of $D_{2n}$ is a cyclic maximal subgroup.

b) Write $D_{2n} = \langle a, b : a^n=b^2=baba=1\rangle$ and let $d$ be a proper divisor of $n$. The subgroups $\langle b\rangle$ and $\langle ba\rangle$ are maximal cyclic subgroups (of order 2) and are contained in the distinct proper subgroups $\langle b, a^d\rangle$ and $\langle ba, a^d\rangle$, respectively. Both of these noncyclic subgroups occur in some irredundant cover, so $\beta(G)>2$. 
\end{proof} 


\section{The case \texorpdfstring{$G = A\rtimes P$}{}}

We first consider the simplest case $G = A\times P$. Equivalently, $G$ is nilpotent. Note $A$ is cyclically embedded in $G$ so $\beta(G) = \beta(P)=2$ by Lemma \ref{basiclemma}.

\begin{proposition}\label{NilpotentProp}
Suppose that $G$ is nilpotent and $\beta(G)=2$. Then $G = P\times \Z_m$ where $(|P|,m)=1$ and either $P = \Z_{p^2} \rtimes \Z_p$ for an odd prime $p$, or else $P = \Z_4\times \Z_2$. 
\end{proposition}

\begin{proof}
From the above remarks, we must show that $\beta(P)=2$ if and only if $P$ is one of the listed $p$-groups. If $\beta(P)=2$ then $P$ has cyclic maximal subgroups, and such $p$-groups have been classified. In particular, either $P$ is a generalized quaternion group or $P = \Z_{p^{n-1}} \rtimes \Z_p $ for some $n\geq 3$ (if $n=2$ then $\beta(P)=1$ by \cite{B}) with semidirect product structures considered below.

First, suppose $P = \langle x, y : x^{2^k} =1, \ y^2 = x^{2^{k-1}}, \ yxy^{-1} = x^{-1} \rangle$ is a generalized quaternion group of order $2^{k+1}$. If $k=2$ then $P=Q_8$ and $\beta(Q_8)=1$ by \cite{B}, so we may assume $k\geq 3$. In this case $\langle y \rangle \subsetneq \langle y, x^2\rangle $ and $\langle xy\rangle \subsetneq \langle xy, x^2 \rangle$. Since $y$ and $xy$ generate maximal cyclic subgroups of $P$ and $\langle y, x^2 \rangle\neq \langle xy, x^2 \rangle$, we have $\beta(P)> 2$. 

If $P = D_{2^n}$ then $\beta(P)\neq 2$ by Lemma \ref{dihedral}.  

Next, suppose $P = \langle a, b: a^{2^{n-1}} = b^p = 1, \  bab^{-1} = a^{-1+2^{n-2}}\rangle$  with $n\geq 4$. Then $\langle b\rangle$ is a maximal cyclic subgroup and $\langle b \rangle \subsetneq \langle a^{2^{n-2}}, b \rangle \subsetneq \langle a^{2^{n-3}}, b \rangle$, so $\beta(P)>2$. 

It was shown in [\cite{B}, example 4.6.] that $\beta(\Z_{p^{n-1}} \times \Z_p)=n-1$, so in the abelian case we must have $n=3$ as claimed. 

Finally, suppose $P = \langle a , b: a^{p^{n-1}} = b^p  = 1, \ bab^{-1} = a^{1+p^{n-2}}\rangle$. Thus $a^p \in Z(G)$. If $n\geq 4$ then $\langle b \rangle \subsetneq \langle b, a^{p^2} \rangle  \subsetneq \langle b, a^p \rangle $, and $\langle b \rangle$ is a maximal cyclic subgroup of $P$, so $\beta(P)> 2$. So assume $n=3$. If $p=2$ then $P=D_8$ which was already ruled out. If $p>2$ then $\langle a^p, b \rangle$ is the unique noncyclic proper subgroup of $P$ and contains $\langle b \rangle$.   
Hence $\beta(P)=2$. 
\end{proof}

Now consider the case where $G=A\rtimes P$ and is not nilpotent. We will show that $\beta(G)\neq 2$. Recall $\langle x_i\rangle \supset A$ is a cyclic maximal subgroup of $G$. Since $G \supsetneq C_G(A) \supset \langle x_i\rangle$, the subgroup $C_G(A)  = \langle x_i\rangle$ is a maximal subgroup of $G$. Since $A\not\subset Z(G)$, $A$ is not cyclically embedded in $G$. By Lemma \ref{basiclemma},  $\beta(P) = \beta(G/A) < \beta(G)=2$. By \cite{B}, either $P$ is cyclic, $P=\Z_p^2$, or $P = Q_8$. 

If $P = Q_8$ then, since $\Aut(A)$ is abelian, $Z(Q_8)= [Q_8, Q_8]$ centralizes $A$. Thus $Z(Q_8)$ is cyclically embedded in $G$. By Lemma \ref{basiclemma} it therefore suffices to consider the cases where $P$ is cyclic or $P = \Z_p^2$. For such $P$ we see that $A\supset [G, G]$ hence the maximal subgroup $C_G(A)$ is normal and so has index $p$. If $P$ is cyclic then its unique index-$p$ subgroup centralizes $A$, and is therefore cyclically embedded. By Lemma \ref{basiclemma}, we may then assume $P= \Z_p$ or $\Z_p^2$. Let $n = |A|$.

Assume first that $P$ is a 2-group so $n$ is odd and $P\in \{ \Z_2, \ \Z_2^2  \}$. If $n$ is a prime power then the only order-2 automorphism of $A$ is inversion. Let $g\in P\setminus C_G(A)$ and write $A = A^+\times A^-$ where $A^{\pm}:=\{ a\in A: gag^{-1}=a^{\pm 1} \}$; these subgroups do not depend on the choice of $g$. It is easy to see that $A^+$ and $A^-$ are Hall subgroups of $A$. Since $A^+$ is cyclically embedded in $G = A\rtimes P = (A^-\rtimes P)\times A^+$, we may assume $A = A^-$. If $P = \Z_2$ then $G = D_{2n}$ and $\beta(D_{2n})\neq2$. by Lemma \ref{dihedral}. If $P= \Z_2^2$ then $G=D_{2n}\times \Z_2 = (\Z_n\rtimes \Z_2) \times \Z_2$. The middle $\Z_2$ factor is a maximal cyclic subgroup contained in the two distinct subgroups $D_{2n}$ and $\Z_2\times \Z_2$, so $\beta(G)= \beta(D_{2n}\times \Z_2) >2$.

Now suppose $p>2$ and $P\in \{\Z_p, \ \Z_p^2\}$. As above, we may assume that no Sylow subgroup of $A$ is centralized by $P$, since any such subgroup would have been a cyclically embedded direct factor of $G$.  

Assume first that $P = \Z_p$, so $P\cap C_G(A)=1$. Then $C_G(P) = Z(G)P$ is a maximal cyclic subgroup of $G$ that contains no Sylow subgroup of $A$. We will proceed by cases on $n$. Suppose $n$ is not a prime power. If $Q_1$ and $Q_2$ are distinct Sylow subgroups of $A$ then $Q_1 C_G(P)$ and $Q_2C_G(P)$ are distinct proper subgroups containing $C_G(P)$, so $\beta(G)>2$. So assume $n$ is a power of a prime $q$. Clearly $Z(G)$ is cyclically embedded and $Z(G/Z(G))=1$. We have $G/Z(G)=\Z_{q^s} \rtimes \Z_p$. If $s=1$ then $\beta(G)=\beta(G/Z(G))= 1$ by \cite{B}. Suppose $s>1$ and let $a$ and $y$ denote generators for $\Z_{q^s}$ and $\Z_p$, respectively. Then $\langle y\rangle$ and $\langle ay\rangle$ are maximal cyclic subgroups of order $p$, and are respectively contained in the distinct subgroups $\langle y, a^q\rangle$ and $\langle ay, a^q\rangle$. Thus $\beta(G)=\beta(G/Z(G)) >2$.

Finally, suppose that $P = \Z_p^2$. If $P_1 = C_G(A)\cap P$ then $P_1$ is a central subgroup of order $p$ but is not cyclically embedded. Thus $2=\beta(G)>\beta(G/P_1) = \beta(A\rtimes \Z_p)$ by Lemma \ref{basiclemma}. From above, $\beta(A\rtimes \Z_p) >1$ unless $n$ is a power of a prime, so we may assume this is the case. Then $\Aut(A)$ is cyclic, so if $y\in P\setminus C_G(A)$ then $C_A(y) = Z(G)\cap A$. So $\langle y, Z(G)\cap A\rangle$ is a maximal cyclic subgroup of $G$ and is contained in the two proper subgroups $A\rtimes \langle y \rangle$ and $Z(G)P$. Therefore $\beta(G)>2$.

\section{The case \texorpdfstring{$G = P\rtimes A$}{}}

Suppose $G= P\rtimes A$ and is not nilpotent. Observe that $C_G(A)=\langle x_i\rangle$ is a cyclic maximal subgroup as in the previous case. Clearly $C_G(A) = C_P(A) \times A$. Let $\Phi(P)$ denote the Frattini subgroup of $P$. We claim first that $\Phi(P) C_P(A) \neq P$. For otherwise $C_P(A)$ surjects onto $P/\Phi(P)$, and if $P/\Phi(P)$ is cyclic then $P$ is cyclic. But then $\Phi(P)$ is the unique maximal subgroup of $P$, and since $C_P(A)\neq P$ we obtain a contradiction. 

Since $\Phi(P)$ is a characteristic subgroup of $P$, it is normal in $G$ and $\Phi(P)C_P(A)\times A$ is a subgroup of $G$ containing $C_G(A)$. Since it is proper by our previous claim, we must have $\Phi(P)\subset C_P(A)$, so $C_P(A)$ is normal in $G$. In the next two propositions we classify those nonnilpotent $G = P\rtimes A$ with $\beta(G)=2$ such that $\Phi(P)$ is trivial and nontrivial, respectively. This will complete the proof of our final classification. 

\begin{proposition} \label{Trivial Phi}
Let $p$ be a prime, $P= \Z_p^k$, and $n>1$ with $(p, n)=1$. Suppose $G = P \rtimes \Z_n$ is not nilpotent, where $C_P(\Z_n)\times \Z_n$ is a cyclic maximal subgroup of $G$ and $C_P(\Z_n)$ is normal in $G$. Then $\beta(G)=2$ if and only if either

a) $k=2$ with $Z(G)=\Z_{n/d}$ for some prime divisor $d$ of $n$, and no proper subgroup of $P$ is normalized by $\Z_n$, or

b) or $k=1$ with $Z(G) = \Z_{n/q^2}$ for a prime divisor $q$ of $n$. 
\end{proposition}

\begin{proof}

We first prove these conditions imply $\beta(G)=2$. In both cases $Z(G)$ is cyclically embedded so we can replace $G$ by $G/Z(G)$ and note that $Z(G/Z(G))=1$ for the given groups. So if a) holds then we may take $n$ to be prime and if b) holds we may take $n$ to be the square of the prime $q$. In case a) let $H = \Z_p\times \Z_p $, and in case b) let $H = \Z_p \rtimes  \Z_q$. In both cases, the maximal cyclic subgroups of $G$ which are not maximal subgroups are all of order $p$, are contained in $H$, and are in no other proper subgroup, so $\beta(G)=2$. 

Now suppose that $\beta(G)=2$. Let $U = C_P(\Z_n)$. Since $(p,n)=1$, there is a subgroup $W\subset P$ with $P = UW$, $U\cap W=1$, and $W\mathrel{\unlhd} G$ (consider $P$ as a vector space over the field with $p$ elements, and decompose it as a representation of $\Z_n$). Then $W\neq 1$ since $G$ is nonabelian. If $W'\subset W$ is a normal subgroup of $G$ then $UW' \rtimes \Z_n \supset U\times \Z_n$. Since $U\times \Z_n$ is a maximal subgroup, we must have $W' = W$ or $W'=1$. Let $1\neq x\in W$. Since $C_G(x)\supset P$, it follows that $C_G(x) = P \rtimes \Z_{n/d}$ for some divisor $d$ of $n$. Let $C$ be a maximal cyclic subgroup of $G$ containing $x$. Since $C\subset C_G(x)$ and $C\cap P = \langle x \rangle$, we have $C = \langle x \rangle \times \Z_{n/d}$. 

We claim $\Z_{n/d}\subset Z(G)$. To see this, let $a$ generate $\Z_n$. Then $a^d \in C_G(a^i xa^{-i})$ for all $i\in \Z$. However, the group generated by all $a^i x a^{-i}$, for $i\in \Z$, is a nontrivial normal subgroup of $G$ contained in $W$. Hence $a^d$ commutes with $W$. Since $a^d$ obviously commutes with $U$, we have $a^d \in Z(G)$ as claimed. 

Now suppose $k\geq 3$. Let $y_1 \in P \setminus \langle x \rangle$ and $y_2 \in P \setminus \langle x, y_1 \rangle$. If $H_i = \langle y_i, x, a^d\rangle \cong \Z_p\times \Z_p \times \Z_{n/d}$ then $H_i\supsetneq C$ and $H_1\neq H_2$. This implies $\beta(G)> 2$, so $1\leq k\leq 2$. 

Suppose $k=2$. If $e$ is a proper divisor of $d$ then $C\subsetneq P \times \Z_{n/d} \subsetneq P\times \Z_{ne/d}$, so $\beta(G)> 2$. Thus $d$ is prime. Now suppose $U\neq 1$, so $U \cong W \cong \Z_p$ and $W = \langle x\rangle$. Then $C \subsetneq P\times \Z_{n/d}$ and $C\subsetneq W \times \Z_n$, forcing $\beta(G)> 2$. Thus $U = 1$, no proper subgroup of $P = W$ is normal in $G$, and $Z(G) = \Z_{n/d}$. 

Finally, suppose $k=1$. Then $P = W = \langle x \rangle$ and $C$ is normal in $G$. If $C$ is a maximal subgroup then $[G:C]=d$ is prime. Clearly $G/Z(G) \cong \Z_p \rtimes \Z_d$ and $Z(G)$ is cyclically embedded. By \cite{B}, this implies $\beta(G)=1$. Therefore $C$ is not a maximal subgroup. Since $\beta(G)=2$ there is a unique proper subgroup containing $C$. It follows that $G/C=\Z_d =\Z_{q^2}$ for some prime divisor $q$ of $n$. The conclusion follows. 
\end{proof}


\begin{proposition}
Let $P$ be a $p$-group and $(p,n)=1$. Suppose $G = P \rtimes \Z_n$ is not nilpotent, where $C_P(\Z_n)\times \Z_n$ is a cyclic maximal subgroup of $G$, $C_P(\Z_n)\supset \Phi(P) \neq \{1\}$, and $C_P(\Z_n)$ is normal in $G$. Then $\beta(G)=2$ if and only if either 

a) $P= \Z_{p^t}$ with $t>1$ and $Z(G) = \Z_{p^{t-1}}\times \Z_{n/q^2}$ for a prime divisor $q$ of $n$, or 

b) $G = (Q_8\rtimes \Z_{3^k})\times \Z_m$ where $(m, 6)=1$.

\end{proposition}

\begin{proof}

Write $P/\Phi(P) = \Z_p^k$ for some $k\geq 1$. We consider cases based on $k$. 

If $k=1$ then $P$ is cyclic. Since $\Phi(P)\subset C_P(A)$, it follows that $\Phi(P)$ is cyclically embedded in $G$. This shows that $G/\Phi(P)$ must a group of the kind considered in part b) of the previous proposition.  

Now suppose that $k\geq 2$. Let $K:=P/\Phi(P) = \Z_p^k$ so $G/\Phi(P)=K\rtimes \Z_n$. The image $I$ of the cyclic maximal subgroup $C_G(\Z_n) = C_P(\Z_n)\times \Z_n$ of $G$ in $G/\Phi(P)$ is still cyclic and maximal. Clearly $I$ is contained in $C_K(\Z_n)\times \Z_n \subset G/\Phi(P)$. If $I = C_K(\Z_n)\times \Z_n$ then $G/\Phi(Q)$ is a group of the type described by the Proposition \ref{Trivial Phi}. If $k\geq 3$ then the proof of that proposition showed $\beta(\Z_p^k \rtimes \Z_n)>2$. Thus $\beta(G)\geq \beta(G/\Phi(P))>2$ and we can conclude that $k=2$. If $I \subsetneq C_K(\Z_n)\times \Z_n $ then maximality of $I$ forces $C_K(\Z_n)\times \Z_n = G/\Phi(P)$. This implies $G/\Phi(P) = K \times \Z_n$ is nilpotent. By Proposition \ref{NilpotentProp} and \cite{B}, if $k\geq 3$ this forces $\beta(G)\geq \beta(G/\Phi(P))> 2$. Thus $k=2$ again.

Lastly, we consider the case $k=2$. The $p$-groups $P$ with cyclic Frattini subgroup have been classified in \cite{BKN}. Those with $P/\Phi(P) = \Z_p^2$ are a substantially more restricted family. Since such $P$ can be generated by two elements, it follows from remark (4) of \cite{BKN} that either $P$ is abelian or $|P|=p^3$. 

If $P$ is abelian then $P= \Z_{p^t}\times \Z_p$ for some $t\geq 2$. Clearly $\Phi(P) = \Z_{p^{t-1}}\subset Z(G)$ but is not cyclically embedded. Thus $\beta(G)>\beta(G/\Phi(P))$ and $G/\Phi(P)=\Z_p^2\rtimes \Z_n$ is nonnilpotent since $G$ is. The nonnilpotent case in \cite{B} implies $\beta(G/\Phi(P))>1$, so $\beta(G)>2$. 

So assume $|P|=p^3$. If $P = \Z_{p^2}\rtimes \Z_p$ is nonabelian, then $\Phi(P)\subset Z(G)$ is not cyclically embedded and the same argument as in the abelian case shows $\beta(G)>2$. If $P$ is the Heisenberg group of order $p^3$, with $p>2$, 
then the exact same argument again shows $\beta(G)>2$.  
Finally, if $P = Q_8$ then $\Aut(Q_8) = S_4$ and so $G=Q_8\rtimes \Z_{3^k}\times \Z_m$ where $(6,m)=1$ and $\Z_{3^k}$ acts on $Q_8$ by an automorphism of order 3. Note $Z(G) = \Z_2\times \Z_{3^{k-1}}\times \Z_m$ is cyclically embedded and $G/Z(G) = (\Z_2\times \Z_2)\rtimes \Z_3$ is a group of the kind considered in Proposition \ref{Trivial Phi} a), so $\beta(G) = \beta(G/Z(G))=2$.  
\end{proof}

\section{Further work}

There a various directions that one could take to strengthen or extend the above results. We consider a few of them here. 

The most obvious step would be a classification of $G$ with given $\beta(G)>2$. We outline a beginning of how this might be achieved for $\beta(G)=3$. If $H_1, \ldots, H_m$ are distinct noncyclic subgroups of $G$ that appear in some irredundant cover, and $\{ \langle x_i\rangle  \}_{i=1}^n$ are the maximal cyclic subgroups of $G$, then for each $1\leq j\leq m$ there is an irredundant cover $\{ H_j  \} \cup \{ \langle x_i \rangle: x_i\not\in H_j  \}$. Every irredundant cover of $G$ has the form $\{  H_{j_k}  \} \cup \{ \langle x_i \rangle: x_i\not\in \bigcup H_{j_k}  \}$ for a (possibly empty) subcollection $\{ H_{j_k}  \}$ of $\{H_j\}_{j=1}^m$.  So $$m+1 \leq \beta(G)\leq 2^m.$$ Clearly $\beta(G)=m+1$ if and only if, for every pair $\{H_j, H_k  \}$ with $j\neq k$, the cover $\{ H_j, H_k \} \cup \{ \langle x_i \rangle: x_i \not\in H_j\cup H_k  \}$ is not irredundant. So if $\beta(G)=3$ then $m=2$, and since $G \neq H_1\cup H_2$, it follows that $G$ contains cyclic maximal subgroups and hence is solvable. 

It is shown in \cite{B} that $\beta(G)$ can take on any possible integer, since $\beta(\Z_{p^n} \times \Z_p)= n$ if $p$ is prime. It should be possible to determine $\beta(G)$ for any abelian $G$ in an explicit fashion. More ambitiously, one may ask what structural information $\beta(G)$ determines when $G$ is a $p$-group, or nilpotent. Indeed, if $|H|$ and $|K|$ are coprime  then $$\beta(H\times K) = \beta(H)+ \beta(K) + \beta(H)\beta(K).$$ This holds because all covers have the form $\{ H_i \times K  \}_{i\in I}$, $\{ H \times K_j  \}_{j\in J}$, or $\{H_i \times K_j  \}_{i\in I,j\in J}$ where $\{H_i\}_{i\in I}$ and $\{K_j\}_{j\in J}$ are covers of $H$ and $K$, respectively. 
In particular, this reduces the computation of $\beta(G)$ for nilpotent groups to the case of $p$-groups. 

One might similarly ask for a determination of $\beta(G)$ for other worthwhile classes of $G$. For example, it is not hard to show directly from the lattice of subgroups that $\beta(D_8)=4$. One may reasonably seek the value of $\beta(D_{2n})$, or more ambitiously $\beta(A_n)$ and $\beta(S_n)$. 

We saw above that if $\beta(G)\leq 3$ then $G$ is solvable. It is natural to seek the smallest value $M$ such that there exists a nonsolvable $G$ with $\beta(G)=M$. Consideration of the case $G=A_5$ suggests $M$ may be substantially larger than 3. One may also ask whether every integer larger than $M$ can be $\beta(G)$ for some nonsolvable $G$, or if there are natural obstructions. 

In another direction, one could constrain the types of irredundant covers that are permitted. For example, one might define the invariant $\beta_N(G)$ to be the number of normal irredundant covers of $G$, meaning those covers preserved by conjugation. Clearly $\beta_N(G)\leq \beta(G)$. As a consequence of the work above, if $\beta(G)\leq 2$ then $\beta_N(G)=\beta(G)$. It is unclear if the converse holds: if $1\leq \beta_N(G)\leq 2$ then may we have $\beta(G)>\beta_N(G)$? Under what conditions on $G$ do we have $\beta(G)=\beta_N(G)$? Can we classify $G$ with small $\beta_N(G)$?

\end{document}